\documentclass[11pt]{article}

\usepackage{graphicx}
\usepackage{graphics}
\usepackage{amssymb}
\usepackage{amsmath}
\usepackage{amsthm}
\usepackage{psfrag}
\usepackage{enumerate}
\usepackage{here}

\usepackage{epsfig}
\usepackage{epsf}
\usepackage{amsfonts}
\usepackage{latexsym}
\usepackage{fancyvrb}
\usepackage[small]{caption}
\usepackage{hyperref}
\usepackage[update,prepend]{epstopdf}

\newtheorem{theorem}{Theorem}
\newtheorem{lemma}[theorem]{Lemma}
\newtheorem{corollary}[theorem]{Corollary}
\newtheorem{fact}[theorem]{Fact}

\newcommand{\etal}{{et~al.}}
\newcommand{\ie}{{i.e.}}
\newcommand{\eg}{{e.g.}}

\newcommand{\old}[1]{{}}
\newcommand{\later}[1]{{}}

\newcommand{\tcup}{{\tt cup}}
\newcommand{\tcap}{{\tt cap}}
\newcommand{\slope}{{\tt slope}}

\newcommand{\e}{\mathrm{e}}
\newcommand{\x}{\mathrm{x}}

\newcommand{\NN}{\mathbb{N}}

\newcommand{\RR}{\mathbb{R}}

\title{\textsc{Convex Polygons in Geometric Triangulations}\footnote{An extended
abstract of this paper appeared in the
{\em Proceedings of the 29th International Symposium on Algorithms
and Data Structures} (WADS 2015), %Victoria, BC, Canada, 2015,
vol.~9214 of LNCS, Springer, 2015, pp.~289--300.}}

\author{Adrian Dumitrescu\thanks{Department of Computer Science,
University of Wisconsin--Milwaukee, USA\@.
Email:~\texttt{dumitres@uwm.edu}}
\and
Csaba D. T\'oth\thanks{Department of Mathematics, California State
  University Northridge, Los Angeles, CA, USA; and
  Department of Computer Science, Tufts University, Medford, MA, USA.
  Email:  \texttt{cdtoth@acm.org}
}}

\begin{document}

\maketitle

\begin{abstract}
We show that the maximum number of convex polygons in a triangulation
of $n$ points in the plane is $O(1.5029^n)$.
This improves an earlier bound of $O(1.6181^n)$ established by
van~Kreveld, L{\"o}ffler, and Pach (2012) and almost matches
the current best lower bound of $\Omega(1.5028^n)$ due to the same authors.
Given a planar straight-line graph $G$ with $n$ vertices, we also show how
to compute efficiently the number of convex polygons in $G$.



\end{abstract}

\section{Introduction} \label{sec:intro}

\paragraph{Convex polygons.}
According to the celebrated Erd\H{o}s-Szekeres theorem~\cite{ESz35},
every set of $n$ points in the plane, no three on a line,
contains $\Omega(\log n)$ points in convex position, and, apart from
the constant factor, this bound is the best possible.
When the $n$ points are in convex position,
then trivially all the $2^n-1$ nonempty subsets are also in convex position.
Erd\H{o}s~\cite{Erd78} proved that the minimum number of subsets in convex position
over all $n$-element point sets with no $3$ collinear points, is $\exp(\Theta(\log^2 n))$.
See also the survey~\cite{MS00} for many other results  related to the
Erd\H{o}s-Szekeres theorem.
\begin{figure}[hbtp]
\centering
\includegraphics[width=.9\textwidth]{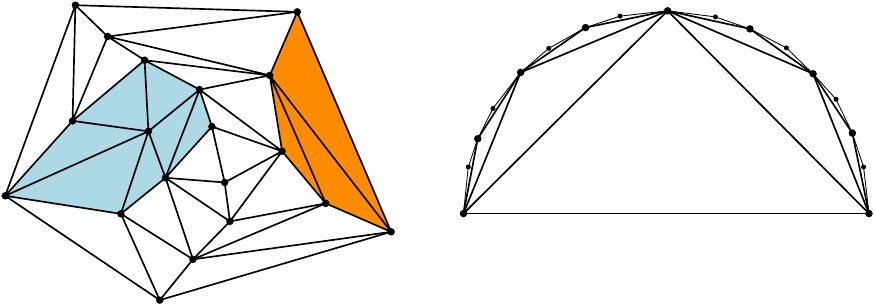}
\caption{Left: A (geometric) triangulation on 19 points; the boundaries
of the two shaded convex polygons are subgraphs of the triangulation.
Right: A triangulation on $2^4+1=17$ points in convex position,
whose dual graph is a full binary tree with $8$ leaves.}
\label{fig:intro}
\end{figure}

Recently, van~Kreveld, L{\"o}ffler, and Pach~\cite{KLP12} posed analogous problems
concerning the number of convex polygons contained in a triangulation of $n$ points
in the plane (as a subgraph); see Fig.~\ref{fig:intro}~(left).
A \emph{convex polygon} is a plane straight-line graph cycle whose interior is convex.
They proved that the maximum number of convex polygons
in a triangulation of $n$ points, no three on a line, is between
$\Omega(1.5028^n)$ and $O(1.6181^n)$. Their lower bound comes from a balanced
binary triangulation on $2^4+1=17$ points shown in Fig.~\ref{fig:intro}~(right).
At the other end of the spectrum, Dumitrescu~\etal~\cite{DLST16} showed that
the \emph{minimum} number of convex polygons in an $n$-vertex triangulation
is $\Theta(n)$. Here we study the \emph{maximum} number
of convex polygons contained in an $n$-vertex triangulation.
Our focus is in the base of the exponent: what is the infimum of $a>0$
such that every $n$-vertex triangulation contains $O(a^n)$ convex polygons?

Throughout this paper we consider planar point sets $S\subset \RR^2$
in \emph{general position}, in the sense that no 3 points are collinear.
A \emph{(geometric) triangulation} of a set $S\subset \RR^2$ is a plane straight-line graph
with vertex set~$S$ such that all bounded faces are triangles that jointly tile the convex
hull of $S$.

\paragraph{Our results.}
We first show that the maximum number of convex polygons in an $n$-vertex
triangulation is attained, up to an $O(n)$-factor, for point sets in convex
position. Consequently, determining the maximum becomes a purely combinatorial problem.
We then prove that the maximum number of convex polygons in a triangulation
of $n$ points in the plane is $O(1.5029^n)$.
This improves an earlier bound of $O(1.6181^n)$ established by
van~Kreveld~\etal~\cite{KLP12} and almost matches
the current best lower bound of $\Omega(1.5028^n)$ due to the same authors
(Theorem~\ref{thm:recurrence} and Corollary~\ref{cor:1} in Subsection~\ref{subsec:paths}).
In deriving the new upper bound, we start in Subsection~\ref{subsec:balanced}
with a careful analysis of a balanced binary triangulation illustrated in
Fig.~\ref{fig:intro}~(right). In Subsection~\ref{subsec:paths} we extend
the analysis to \emph{all} triangulations on $n$ points in convex position.
In Section~\ref{sec:algo} we focus on an algorithmic problem:
given a planar straight-line graph $G$ with $n$ vertices, determine the number of
convex polygons in $G$. Our main results are summarized in the following.

\begin{theorem}\label{thm:main}
The maximum number of convex polygons in a triangulation
of $n$ points in the plane is $O(1.5029^n)$.
\end{theorem}

\begin{theorem}\label{thm:algo}
Given a planar straight-line graph $G$ with $n$ vertices,
the number of convex polygons in $G$ can be computed in $O(n^2)$ time.
The convex polygons can be enumerated in an additional $O(1)$-time per edge.
\end{theorem}

\paragraph{Related work.}
In this paper we derive new upper and lower bounds on the maximum number of convex cycles
in a straight-line triangulation of $n$ points in the plane.
In another article that can be included in the same general theme,
Dumitrescu, Mandal, and T\'oth~\cite{DMT16} recently showed that the (maximum) number of
monotone paths in a geometric triangulation of $n$ points in the plane is $O(1.7864^n)$;
this improves an earlier upper bound of $O(1.8393^n)$ in~\cite{DLST16}; the current best
lower bound, $\Omega(1.7003^n)$, appears in~\cite{DLST16}.

Convex polygons and monotone paths can be defined geometrically---in terms of
angles or coordinates.
Analogous problems have been previously studied for cycles, spanning cycles,
spanning trees, and matchings~\cite{BKK+07} in $n$-vertex edge-maximal planar
graphs---that are defined in purely graph theoretic terms.
For plane straight-line graphs, previous research focused on the
maximum number of (noncrossing) configurations such as plane graphs, spanning trees,
spanning cycles, triangulations, and others, over all $n$-element point sets
in the plane~\cite{AHV+06,ACNS82,DSST13,GNT00,RSW08,SS11,SS13,SSW12,SW06a};
see also the two surveys~\cite{DT12,She14}. Early upper bounds in this area were obtained
by multiplying the maximum number of triangulations on $n$ point in the plane
with the maximum number of desired configurations in an $n$-vertex triangulation,
based on the fact that every planar straight-line graph can be augmented into a triangulation.

The problem of finding the largest convex polygon in a nonconvex container
has a long history in computational geometry. Polynomial-time algorithms are known
in the plane for the problems of computing a convex polygon with the maximum area or
the maximum number of vertices contained in a given simple polygon
with $n$ vertices~\cite{AKLS11,CY86,Goo81} (called the \emph{potato peeling} problem);
or spanned by a given set of $n$ points~\cite{EORW92}.

\section{Convex polygons in a triangulation} \label{sec:convex}

\paragraph{Section outline.}
We reduce the problem of determining the maximum number of convex polygons
in an $n$-vertex triangulation (up to polynomial factors) to triangulations
of $n$ points in convex position (Theorem~\ref{thm:cx},
Section~\ref{subsec:reduction1}). We further reduce the problem
to counting convex \emph{paths} between two adjacent hull vertices
in a triangulation (Lemma~\ref{lem:leaf}, Subsection~\ref{subsec:reduction2}).
We first analyze the number of convex paths in a balanced binary triangulation,
which gives the current best lower bound~\cite{KLP12}
(Theorem~\ref{thm:lambda-ub}, Subsection~\ref{subsec:balanced}).
The new insight gained from this analysis is then generalized to
derive an upper bound for all $n$-vertex triangulations
(Theorem~\ref{thm:recurrence} and Corollary~\ref{cor:1}, Subsection~\ref{subsec:paths}).

\subsection{Reduction to convex position} \label{subsec:reduction1}

For a plane straight-line graph $G$, let $C(G)$ denote the number of convex
polygons in $G$. For an integer $n\geq 3$, let $C(n)$ be the maximum of $C(G)$
over all plane straight-line graphs $G$ of $n$ points in the plane;
and let $C_\x(n)$ be the maximum of $C(G)$ over all plane straight-line graphs $G$
of $n$ points \emph{in convex position}. It is clear that $C_\x(n)\leq C(n)$ for
every $n\geq 3$. The main result of this subsection is the following.

\begin{theorem}\label{thm:cx}
For every $n\geq 3$, we have $C(n)\leq (2n-5) \, C_\x(n)$.
\end{theorem}

Theorem~\ref{thm:cx} is an immediate consequence of the following lemma.

\begin{lemma}\label{lem:cx}
Let $T$ be a triangulation on a set $S$ of $n$ points in the plane,
and let $f$ be a (triangular) face of $T$. Then there exists
a set $S'$ of $n$ points in the plane in convex position
and a triangulation $T'$ of $S'$ such that
the number of convex polygons in $T$ containing $f$ is at most $C(T')$.
\end{lemma}

\begin{proof}
We construct a point set $S'$ in convex position, a plane straight-line graph
$G'$ on $S'$, and then give an injective map from the set of convex
polygons in $T$ that contain $f$ into the set of convex polygons of $G'$.
For any triangulation $T'$ of $G'$, we have $C(G')\leq C(T')$.

Let $o$ be a point in the interior of $f$, not contained in any line determined by $S$;
and let $O$ be a circle centered at $o$ that contains all points in $S$ in its interior;
refer to Fig.~\ref{fig:convex}.
For each point $p\in S$, let $p'$ be the intersection point of the ray
$\overrightarrow{op}$ with $O$. Let $S'=\{p': p\in S\}$.

\begin{figure}[htbp]
\centering
\includegraphics[width=.98\textwidth]{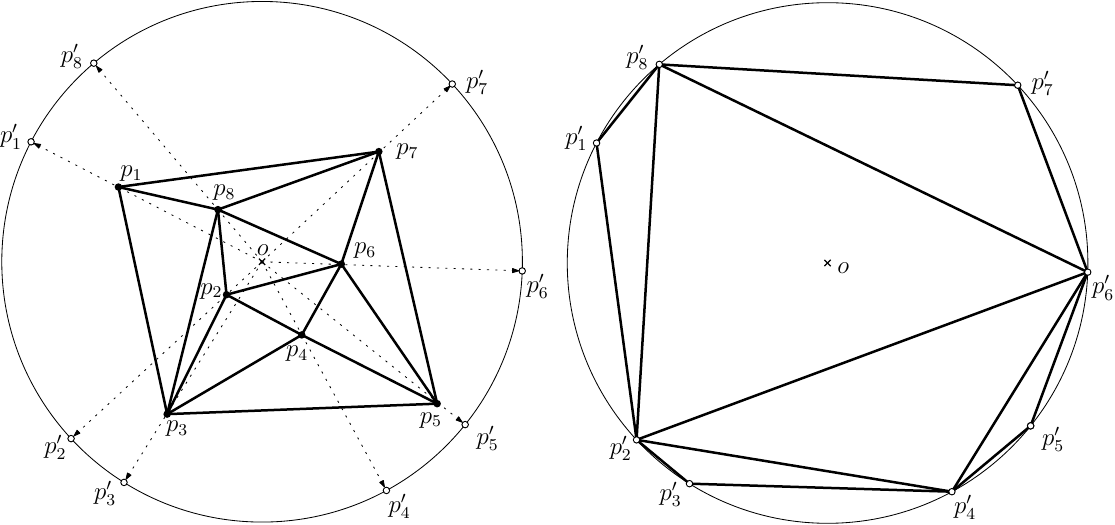}
\caption{A triangulation $T$ on the point set $\{p_1,\ldots ,p_8\}$ (left) is mapped
to a triangulation $G'$ on the point set $\{p_1',\ldots ,p_8'\}$ in convex position (right).
This induces an injective map from the convex polygons in $T$ containing $o$
to convex polygons in $G'$. For example, $(p_3,p_5,p_6,p_8)$ is mapped to
$(p_3',p_4',p_5',p_6',p_8',p_2')$.}
\label{fig:convex}
\end{figure}

We now construct a plane straight-line graph $G'$ on the point set $S'$.
For two points $p', q'\in S'$, insert an edge $p'q'$ in $G'$ if and only
if there is a triangle $\Delta{oab}$ whose interior is disjoint from $S$
such that segment $ab$ is contained in an edge of $T$,
point $p$ lies on segment $oa$, and $q$ lies on $ob$.
Intuitively, the rays $\overrightarrow{op}$ and $\overrightarrow{oq}$ cross a common
edge of $T$ at $a$ and $b$, respectively, and segment $ab$ is ``mapped'' to $p'q'$.

Note that no two edges in $G'$ cross each other. Indeed, suppose to
the contrary that edges $p_1'q_1'$ and $p_2'q_2'$ cross in $G'$.
By construction, there are triangles $\Delta{oa_1b_1}$ and $\Delta{oa_2b_2}$ that induce
$p_1'q_1'$ and $p_2'q_2'$, respectively. We may assume without loss of generality
that both $\Delta{oa_1b_1}$ and $\Delta{oa_2b_2}$ are oriented counterclockwise,
and $\overrightarrow{oa_2}$ enters the interior of $\Delta{oa_1b_1}$
(refer to Fig.~\ref{fig:oab}).
Since $a_1b_1$ and $a_2b_2$ do not cross (they may be collinear),
segment $oa_2$ lies in $\Delta{oa_1b_1}$ or segment $ob_1$ lies in $\Delta{oa_2b_2}$.
That is, one of $\Delta{oa_1b_1}$ and $\Delta{oa_2b_2}$
contains a point from $S$, contradicting the assumption that both triangles are empty.

\begin{figure}[H]
\centering
\includegraphics[width=.85\textwidth]{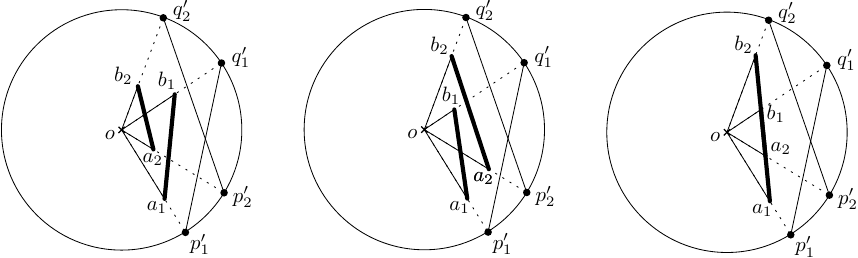}
\caption{Assume that edges $p_1'q_1'$ and $p_2'q_2'$ cross in $G'$;
both $\Delta{oa_1b_1}$ and $\Delta{oa_2b_2}$ are oriented counterclockwise;
and ray ${oa_2}$ enters the interior of $\Delta{oa_1b_1}$.
Then segment $oa_2$ lies in $\Delta{oa_1b_1}$ (left),
or segment $ob_1$ lies in $\Delta{oa_2b_2}$ (middle),
or both (right).}
\label{fig:oab}
\end{figure}

Finally, we define an injective map from the convex polygons of $T$
that contain $o$ into the convex polygons of $G'$. To define this
map, we first map every edge of $T$ to a path in $G'$. Let $pq$ be an edge in $T$,
and assume without loss of generality that $\Delta{opq}$ is oriented counterclockwise.
We map the edge $pq$ to the path $(p'=r_0',r_1',\ldots , r_k',r_{k+1}'=q')$,
where $(r_1,\ldots , r_k)$ is the sequence of all points in $S$ lying in the
interior of $\Delta{opq}$ in counterclockwise order around $o$.
All edges of this path are present in $G'$, since $\Delta{or_ir_{i+1}}$
is empty of vertices in $S$, and both rays $\overrightarrow{or_i}$
and $\overrightarrow{or_{i+1}}$ intersect segment $pq$, for $i=0,\ldots, k$.
A convex polygon $A=(p_1,\ldots , p_k)$ containing $o$ in $T$ is
mapped to the convex polygon $A'$ in $G'$ obtained by concatenating
the images of the edges $p_1p_2,\ldots ,p_{k-1}p_k$, and $p_kp_1$.
Consequently, the vertex set of $A'$ consists of the images of all
points in $S$ that lie on the boundary or in the interior of $A$.

It remains to show that the above mapping is injective on the convex
polygons of $T$ that contain~$o$.
Consider a convex polygon $A'=(p_1', \ldots ,p_k')$ in $G'$ that
is the image of some convex polygon in $T$ containing $o$.
Then the preimage $A$ must be a convex polygon in $T$
for which $\{p_1,\ldots ,p_k\}$ is the set of points in $S$
that lie on the boundary or in the interior of $A$.
Consequently, $A$ is the boundary of the convex hull of
$\{p_1,\ldots ,p_k\}$, that is, $A'$ has a unique preimage.
\end{proof}

\paragraph{Proof of Theorem~\ref{thm:cx}.}
Let $T$ be a (geometric) triangulation with $n$ vertices. Every $n$-vertex triangulation
has at most $2n-4$ faces (including the outer face), and hence at most $2n-5$ bounded faces.
By Lemma~\ref{lem:cx}, each bounded face $f$ of $T$ lies in the interior of at most $C_\x(n)$
convex polygons contained in $T$. Summing over all bounded faces $f$, the number
of convex polygons in $T$ is bounded by $C(T)\leq (2n-5) \, C_\x(n)$, as required.
\hfill$\Box$

\subsection{Reduction to convex paths} \label{subsec:reduction2}

A  \emph{convex path} is a \emph{simple} polygonal chain $(p_1,\ldots,p_m)$
that makes a right turn at each interior vertex $p_2, \ldots, p_{m-1}$. Let
$P(n)$ denote the maximum number of convex paths between two consecutive hull vertices
in a triangulation of $n$ points in convex position.
A convex path from $a$ to $b$ is either a direct path consisting
of a single segment $ab$, or a path that can be decomposed into
two convex subpaths sharing a common endpoint $c$,
where $\Delta{abc}$ is a counterclockwise triangle incident to $ab$;
see Fig.~\ref{fig:paths}.
\begin{figure}[hbtp]
\centering
\includegraphics[width=.66\textwidth]{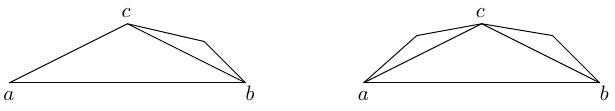}%f4.eps}
\caption{Convex paths in a triangulation.
Left: $P(4)= P(2) \, P(3) + 1 = 3$.
Right: $P(5)= P(3) \, P(3) + 1 = 5$. }
\label{fig:paths}
\end{figure}

Thus $P(n)$ satisfies the following recurrence for $n\geq 3$, with
initial values $P(2)=1$ and $P(3)=2$.
\begin{equation} \label {eq:1}
P(n)= \!\max_
{\substack {n_1 + n_2=n+1\\ n_1, n_2 \geq 2}} \left\{ P(n_1) \, P(n_2) + 1 \right\}.
\end{equation}

\paragraph{Remark.}
The values of $P(n)$ for $2\leq n\leq 18$ are shown in Table~\ref{table}.
It is worth noting that $P(n)$ need not be equal to
$P(\lfloor \frac{n+1}{2}\rfloor)\, P(\lceil \frac{n+1}{2}\rceil)+1$;
for instance, $P(7)=P(3)\, P(5)+1>P(4)\, P(4)+1$.
That is, the balanced partition of a convex $n$-gon into two subpolygons
does not always maximize $P(n)$.
However, we have $P(n)=P(\frac{n+1}{2})\, P(\frac{n+1}{2})+1$ for $n=2^k+1$ and
$k=1,2,3,4$; these are the values relevant for the (perfectly) balanced binary
triangulation discussed in Subsection~\ref{subsec:balanced}.

Let $ab$ be a hull edge of a triangulation $T$ on $n$ points in convex
position. Suppose that $ab$ is incident to a counterclockwise triangle
$\Delta{abc}$. The edges $ac$ and $bc$ decompose $T$ into three
triangulations $T_1$, $\Delta{abc}$ and $T_2$, of size $n_1$, $3$ and
$n_2$, where $n_1+n_2=n+1$. A convex polygon in $T$ is either (i)~contained in $T_1$;
or (ii)~contained in $T_2$; or (iii)~the union of $ab$ and a convex path from
$a$ to $b$ that passes through $c$; see Fig.~\ref{fig:paths}.
Consequently, $C_\x(n)$, the maximum number of convex polygons contained in a
triangulation of $n$ points in convex position, satisfies the following recurrence:
\begin{equation} \label {eq:13}
C_\x(n) = \!\max_
{\substack {n_1 + n_2=n+1\\ n_1, n_2 \geq 2}} \left\{ P(n_1) \, P(n_2)
+ C_\x(n_1) + C_\x(n_2) \right\}
\end{equation}
for $n\geq 3$, with initial values $C_\x(2)=0$ and $C_\x(3)=1$.
The values of $C_\x(n)$ for $2\leq n \leq 9$ are displayed in Table~\ref{table}.

\begin{table*}[ht]
\begin{center}
\begin{tabular}{||c||c|c|c|c|c|c|c|c|c|c|c|c|c|c|c|c|c|c||} \hline
$n$ & 2 & 3 & 4 & 5 & 6 & 7 & 8 & 9 & 10 & 11 & 12 & 13 & 14 & 15 & 16 &17 &18 \\ \hline
$P(n)$  & 1 & 2 & 3 & 5 & 7 & 11 & 16 & 26 & 36 & 56 & 81 & 131 & 183
  & 287 & 417 & 677 & 937\\ \hline
$C_\x(n)$ & 0 & 1 & 3 & 6 & 11 & 18 & 29 & 45 & & & & & & & & & \\ \hline
\end{tabular}
\end{center}
\caption{$P(n)$ and $C_\x(n)$ for small $n$.}
\label{table}
\end{table*}
\begin{lemma}\label{lem:leaf}
We have $C_\x(n) \leq \sum_{k=2}^{n-1} P(k)$. Consequently, $C_\x(n) \leq n \, P(n)$.
\end{lemma}
\begin{proof}
  We first prove the inductive inequality:
\begin{equation} \label{eq:6}
C_\x(n) \leq P(n-1)+ C_\x(n-1).
\end{equation}
Let $T$ be an arbitrary triangulation of a set $S$ of $n$ points in the plane.
Consider the dual graph $T^*$ of $T$, with a vertex for each triangle in $T$
and an edge for every pair of triangles sharing an edge. It is well known that
if the $n$ points are in convex position, then $T^*$ is a tree.
Let $\Delta{abc}$ be a triangle corresponding to a leaf in $T^*$,
sharing a unique edge, say $e=ab$, with other triangles in $T$; see Fig.~\ref{fig:leaf}.
\begin{figure}[hbtp]
\centering
\includegraphics[width=.3\textwidth]{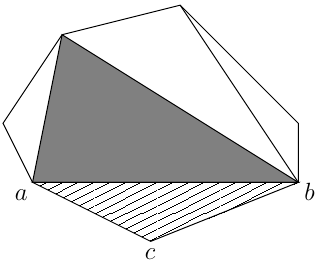}%f5.eps}
\caption{Proof of Lemma~\ref{lem:leaf}.}
\label{fig:leaf}
\end{figure}

We distinguish two types of convex polygons contained in $T$:
(i) those containing both edges $ac$ and $cb$, and
(ii) those containing neither $ac$ nor $cb$.
Observe that the number of convex polygons of type (i) is at most $P(n-1)$,
since any such polygon can be decomposed into the path $(b,c,a)$ and another
path connecting $a$ and $b$ in the subgraph of $T$ induced by $S \setminus \{c\}$.
Similarly, the number of convex polygons of type (ii) is at most $C_\x(n-1)$, since
they are contained in the subgraph of $T$ induced by $S \setminus \{c\}$.
Altogether we have $C_\x(n) \leq P(n-1)+ C_\x(n-1)$ and~\eqref{eq:6} is
established.

The recursion~\eqref{eq:6} yields
$$ C_\x(n) \leq \sum_{k=2}^{n-1} P(k),$$
which is the first inequality in the lemma.
Since $P(k) \leq P(k+1)$, for every $k \geq 2$,
the second inequality in the lemma follows from the first.
\end{proof}

\subsection{The balanced binary triangulation} \label{subsec:balanced}

In this subsection we revisit the balanced binary triangulation
on $n$ points, used by van Kreveld, L\"offler and Pach~\cite[Sec.~3.1]{KLP12} in deriving their
lower bound $\Omega(1.5028^n)$ on the number of convex polygons contained in a triangulation.
For a fixed $k\in \NN$, let $T_k$
be the triangulation on $2^k+1$ points, say, on a circular arc, such that
the dual graph $T_k^*$ is a balanced binary tree; see Fig.~\ref{fig:intro}~(right).
Denote by $\lambda_k$ the number of convex paths between the leftmost and the rightmost
vertex in $T_k$. As noted in~\cite{KLP12}, $\lambda_k$ satisfies the following recurrence:
\begin{equation} \label{eq:10}
\lambda_{k+1}=\lambda_k^2+1, \text{ for } k \geq 0, \hspace{1cm} \lambda_0=1.
\end{equation}
The values of $\lambda_k$ for $0 \leq k \leq 5$ are shown in Table~\ref{table2}.
Note that $\lambda_k=P(2^k+1)$ for these values.
\begin{table*}[ht]
\begin{center}
\begin{tabular}{||c||c|c|c|c|c|c||} \hline
$k$          & 0 & 1 & 2 & 3 &  4  & 5\\ \hline
$\lambda_k$  & 1 & 2 & 5 & 26 & 677& 458330\\ \hline
\end{tabular}
\end{center}
\caption{The values of $\lambda_k$ for small $k$.}
\label{table2}
\end{table*}

The authors constructed a triangulation of $n=m2^k+1$ points, for $m\in \NN$,
by concatenating $m$ copies of $T_k$ along a common circular arc,
where consecutive copies share a vertex, and by triangulating the
convex hull of the $m$ chords arbitrarily to obtain a triangulation of the $n$ points.
The lower bound $\Omega(1.5028^n)$ in~\cite[Sec.~3.1]{KLP12}
is obtained by setting $k=4$. This construction yields
$$ C(n) \geq C_\x(n) \geq \lambda_4 ^{(n-1)/16} =
\lambda_4^{-1/16} \left(\lambda_4^{1/16}\right)^n =\Omega(1.5028^n), $$
for every $n=2^4 m+1$.

Since the above lower bound directly depends on $\lambda_k$, we next examine this dependency.
Obviously~\eqref{eq:10} implies that the sequence
$(\lambda_k)^{1/2^k}$ is strictly increasing; as such,
$\lambda_k \geq 1.5028^{2^k}$ for every $k \geq 4$.
In this subsection (Theorem~\ref{thm:lambda-ub}), we establish an almost matching upper bound
$\lambda_k \leq 1.50284^{2^k}$, or equivalently, $(\lambda_k)^{1/2^k}\leq 1.50284$
for every $k \geq 0$. The tools used here streamline the way to Subsection~\ref{subsec:paths},
where we prove an upper bound of $O(1.50285^n)$ on the number of convex polygons contained
in any triangulation of $n$ points.

We start by bounding $\lambda_k$ from above by a product.
To this end we frequently use the standard inequality $1+x \leq \e^x$,
where $\e$ is the base of the natural logarithm.

\begin{lemma}\label{lem:ind}
For $k\in \NN$, we have
\begin{equation} \label{eq:7}
\lambda_k \leq 2^{2^{k-1}} \prod_{i=1}^{k-1} \left( 1+ \frac{1}{2^{2^i}} \right)^{2^{k-1-i}}.
\end{equation}
\end{lemma}
\begin{proof}
Observe that~\eqref{eq:10} implies $\lambda_k \geq 2^{2^{k-1}}$ for $k\geq 1$.
We thus have
\begin{align*}
\lambda_0 &=1, \\
\lambda_1 &=1^2 +1=2, \\
\lambda_2 &=\lambda_1^2 +1= 2^2 \left( 1 + \frac{1}{2^2} \right), \\
\lambda_3 &=\lambda_2^2 +1 \leq 2^4 \left( 1 + \frac{1}{2^2} \right)^2
 \left( 1 + \frac{1}{2^4} \right), \\
\vdots
\end{align*}
We prove~\eqref{eq:7} by induction on $k$. The base case $k=1$ is verified as shown above.
For the induction step, we assume that inequality~\eqref{eq:7} holds for $k$ and show that it
holds for $k+1$. Indeed, we have
\begin{align*}
\lambda_{k+1} &=\lambda_k^2 +1 \leq
2^{2^{k}} \prod_{i=1}^{k-1} \left( 1+ \frac{1}{2^{2^i}} \right)^{2^{k-i}} +1 \\
&\leq 2^{2^{k}} \prod_{i=1}^{k-1} \left( 1+ \frac{1}{2^{2^i}} \right)^{2^{k-i}}
\left( 1+ \frac{1}{2^{2^k}} \right)
= 2^{2^{k}} \prod_{i=1}^{k} \left( 1+ \frac{1}{2^{2^i}} \right)^{2^{k-i}},
\end{align*}
as required.
\end{proof}

The following sequence is instrumental for manipulating the exponents in~\eqref{eq:7}. Let
\begin{equation}\label{eq:alpha}
\alpha_k=2^k + k+1 \text{~~~~for } k \geq 1.
\end{equation}
That is, $\alpha_1=4$, $\alpha_2=7$, $\alpha_3=12$, $\alpha_4=21$, $\alpha_5=38$, etc.
The way this sequence appears will be evident in Lemma~\ref{lem:withalpha},
and subsequently, in the chains of inequalities~\eqref{eq:14} and~\eqref{eq:15}
in the proof of Theorem~\ref{thm:recurrence}.
We next prove the following.

\begin{lemma}\label{lem:withalpha}
For $k\in \NN$, we have
\begin{equation} \label{eq:8}
 \lambda_k \leq 2^{2^{k-1}} \exp \left( 2^k \sum_{i=1}^{k-1} 2^{-\alpha_i} \right).
\end{equation}
\end{lemma}

\begin{proof}
The inequality $1+x \leq \e^x$ in~\eqref{eq:7} yields:
\begin{align*}
\lambda_k &\leq 2^{2^{k-1}} \prod_{i=1}^{k-1} \left( 1+ \frac{1}{2^{2^i}} \right)^{2^{k-1-i}}
\leq 2^{2^{k-1}} \exp \left( \sum_{i=1}^{k-1} 2^{k-1-i -2^i} \right) \\
&=2^{2^{k-1}} \exp \left( \sum_{i=1}^{k-1} 2^{k-\alpha_i} \right)
=2^{2^{k-1}} \exp \left( 2^k \sum_{i=1}^{k-1} 2^{-\alpha_i} \right),
\end{align*}
as required.
\end{proof}

Taking the $1/2^k$ root in~\eqref{eq:8} yields a first rough approximation
(details in Fact~\ref{fact:1} of Appendix~\ref{sec:numeric}):
\begin{equation*}
(\lambda_k)^{1/2^k} \leq 2^{2^{k-1}/2^k}
\exp \left( 2^k/2^k \sum_{i=1}^{k-1} 2^{-\alpha_i} \right)
\leq 2^{1/2} \exp{ \left( \sum_{i=1}^\infty 2^{-\alpha_i} \right)}
\leq 1.5180,
\end{equation*}
To obtain a sharper estimate, we keep the first few terms in the
sequence as they are, and only introduce approximations for latter terms.

\begin{theorem} \label{thm:lambda-ub}
For every $k\in \NN$, we have $\lambda_k \leq 1.50284^{2^k}$.
\end{theorem}
\begin{proof}
From~\eqref{eq:10}, for every $k \geq 0$ we have
\begin{align*}
\lambda_{k+1} &=\lambda_k^2 +1 =\lambda_k^2 \left( 1 + \frac{1}{\lambda_k^2} \right)
\leq \lambda_k^2 \left( 1 + \frac{1}{2^{2^k}} \right), \\
\lambda_{k+2} &=\lambda_{k+1}^2 +1 =\lambda_{k+1}^2 \left( 1 + \frac{1}{\lambda_{k+1}^2} \right)
\leq \lambda_k^4 \left( 1 + \frac{1}{2^{2^k}} \right)^2
\left( 1 + \frac{1}{2^{2^{k+1}}} \right), \\
\lambda_{k+3} &=\lambda_{k+2}^2 +1 =\lambda_{k+2}^2 \left( 1 + \frac{1}{\lambda_{k+2}^2} \right)
\leq \lambda_k^8 \left( 1 + \frac{1}{2^{2^k}} \right)^4
\left( 1 + \frac{1}{2^{2^{k+1}}} \right)^2
\left( 1 + \frac{1}{2^{2^{k+2}}} \right), \\
\vdots
\end{align*}
For every $k \geq 0$ and $i \geq 1$ we have
\begin{align*}
\lambda_{k+i} &=\lambda_{k+i-1}^2 +1 =\lambda_{k+i-1}^2 \left( 1 + \frac{1}{\lambda_{k+i-1}^2} \right)
\leq (\lambda_k)^{2^{i}}  \prod_{j=1}^{i} \left( 1+ \frac{1}{2^{2^{k+j-1}}} \right)^{2^{i-j}}\\
&\leq (\lambda_k)^{2^{i}}  \exp \left( \sum_{j=1}^{i} 2^{i+k-\alpha_{k+j-1}} \right)
= (\lambda_k)^{2^{i}} \exp \left( 2^{i+k} \sum_{j=1}^{i} 2^{-\alpha_{k+j-1}} \right).
\end{align*}

Consequently,
\begin{equation*}
(\lambda_{k+i})^{1/2^{k+i}}
\leq (\lambda_k)^{2^{i}/2^{i+k}} \exp \left( \sum_{j=1}^{i} 2^{-\alpha_{k+j-1}} \right)
= (\lambda_k)^{1/2^k} \exp \left( \sum_{j=1}^{i} 2^{-\alpha_{k+j-1}} \right).
\end{equation*}

For $i \geq 1$ and $k=4$ the above inequality is
\begin{equation*}
(\lambda_{4+i})^{1/2^{4+i}}
  \leq (\lambda_4)^{1/2^4} \exp \left( \sum_{j=1}^{i} 2^{-\alpha_{4+j-1}} \right)
=(\lambda_4)^{1/2^4} \exp \left( \sum_{j=4}^{i+3} 2^{-\alpha_j} \right).
  \end{equation*}

Put $k=i+4 \geq 5$; and so the following holds for $k \geq 5$:
\begin{align} \label{eq:11}
(\lambda_k)^{1/2^k} &\leq (\lambda_4)^{1/2^4} \exp \left(\sum_{i=4}^{k-1} 2^{-\alpha_i} \right)
= 677^{1/16} \exp \left(\sum_{i=4}^{k-1} 2^{-\alpha_i} \right) \nonumber \\
&\leq 677^{1/16} \exp \left(\sum_{i=4}^\infty 2^{-\alpha_i} \right) \leq 1.50284.
\end{align}
The last inequality in the above chain is Fact~\ref{fact:2} in Appendix~\ref{sec:numeric}.
The inequality $(\lambda_k)^{1/2^k} \leq 1.50284$ also holds for
$0 \leq k \leq 4$, and thus for all $k \geq 0$, as required
(recall that  the sequence $(\lambda_k)^{1/2^k}$ is strictly increasing).
\end{proof}

\subsection{Convex paths in a triangulation of a convex point set} \label{subsec:paths}

In this subsection we show that the maximum number of convex paths
between two adjacent vertices in a triangulation of $n$ points in convex position
is $O(1.50284^n)$, that is, $P(n)=O(1.50284^n)$.
In the main step, a complex proof by induction yields the following.

\begin{theorem} \label{thm:recurrence}
Let $n \geq 2$, where $2^k +1 \leq n \leq 2^{k+1}$. Then
\begin{equation} \label {eq:2}
P(n)^{\frac{1}{n-1}}  \leq (P(17))^{1/16} \exp{ \left( \sum_{i=4}^{k-1} 2^{-\alpha_i} \right)}
=677^{1/16} \exp{ \left( \sum_{i=4}^{k-1} 2^{-\alpha_i} \right)}.
\end{equation}
\end{theorem}
\begin{proof}
We prove the inequality by induction on $n$. The base cases $2 \leq n \leq 32$
are satisfied; this is verified by direct calculation in
Facts~\ref{fact:basis1} and~\ref{fact:basis2} of Appendix~\ref{sec:numeric}:
\begin{align*}
\max_{2\leq n\leq 16} P(n)^{\frac{1}{n-1}} &= P(9)^{1/8} =26^{1/8} =1.50269\ldots. \\
\max_{17\leq n\leq 32} P(n)^{\frac{1}{n-1}} &= P(17)^{1/16} =677^{1/16} =1.50283\ldots.
\end{align*}
Since $0 \leq k \leq 4$, it follows that
$$ \max_{2\leq n\leq 32} P(n)^{\frac{1}{n-1}} = P(17)^{1/16} =677^{1/16}
= 677^{1/16} \exp{ \left( \sum_{i=4}^{k-1} 2^{-\alpha_i} \right)}. $$

Assume now that $n \geq 33$, hence $k \geq 5$, and that
the required inequality holds for all smaller values.
We will show that for all pairs $n_1,n_2 \geq 2$ with $n_1+n_2=n+1$,
the expression $P(n_1) \, P(n_2)+1$ is bounded from above as required.
Note that since  $n_1+n_2=n+1$, we have $n_1,n_2 \leq n-1$, so using
the induction hypothesis for $n_1$ and $n_2$ is justified.
It suffices to consider pairs with $n_1 \leq n_2$.
We distinguish two cases:

\medskip
\emph{Case 1: $2 \leq n_1 \leq 16$.}
Since $n \geq 33$, we have $18 \leq n_2 \leq n-1$.
By the induction hypothesis we have
$$ P(n_2)^{1/(n_2-1)}  \leq
677^{1/16} \exp{ \left( \sum_{i=4}^{k-1} 2^{-\alpha_i} \right)}. $$
Further,
\begin{align*}
P(n) &\leq P(n_1) \, P(n_2) + 1 \\
&\leq P(n_1) \, 677^{\frac{n_2-1}{16}} \exp{ \left( (n_2-1) \sum_{i=4}^{k-1}
  2^{-\alpha_i} \right)} +1 \\
&\leq P(n_1) \, 677^{\frac{n_2-1}{16}} \exp{ \left( (n_2-1) \sum_{i=4}^{k-1} 2^{-\alpha_i} \right)}
\left( 1 + (P(n_1))^{-1} \, 677^{-\frac{n_2-1}{16}} \right) \\
&\leq P(n_1) \, 677^{\frac{n_2-1}{16}} \exp{ \left( (n_2-1) \sum_{i=4}^{k-1} 2^{-\alpha_i} \right)}
\exp{ \left( (P(n_1))^{-1} \, 677^{-\frac{n_2-1}{16}} \right)}.
\end{align*}

To settle Case 1, it suffices to show that
$$ P(n_1) \, 677^{\frac{n_2-1}{16}} \exp{ \left( (n_2-1) \sum_{i=4}^{k-1} 2^{-\alpha_i} \right)}
\exp{ \left( (P(n_1))^{-1} \, 677^{-\frac{n_2-1}{16}} \right)} \leq $$
$$\leq  677^{\frac{n-1}{16}} \exp{ \left( (n-1) \sum_{i=4}^{k-1}
  2^{-\alpha_i} \right)}, $$
or equivalently,
\begin{equation} \label{eq:5}
P(n_1) \, \exp{ \left( (P(n_1))^{-1} \, 677^{-\frac{n_2-1}{16}} \right)}
\leq 677^{\frac{n_1-1}{16}} \exp{ \left( (n_1-1) \sum_{i=4}^{k-1} 2^{-\alpha_i} \right)}.
\end{equation}

We have $n_1 + n_2=n+1$, hence $n_2-1=n-n_1 \geq 33 -n_1$.
By Fact~\ref{fact:Case1} in Appendix~\ref{sec:numeric}, the following inequality
holds for $2 \leq n_1 \leq 16$:
\begin{equation} \label{eq:12}
P(n_1) \, \exp{ \left( (P(n_1))^{-1} \, 677^{-\frac{33-n_1}{16}} \right)}
\leq 677^{\frac{n_1-1}{16}}.
\end{equation}

Now~\eqref{eq:12} in conjunction with $n_2-1 \geq 33 -n_1$ yields
\begin{align*}
P(n_1) \, \exp{ \left( (P(n_1))^{-1} \, 677^{-\frac{n_2-1}{16}} \right)}
&\leq P(n_1) \, \exp{ \left( (P(n_1))^{-1} \, 677^{-\frac{33-n_1}{16}} \right)}\\
&\leq 677^{\frac{n_1-1}{16}} \leq
677^{\frac{n_1-1}{16}} \exp{ \left( (n_1-1) \sum_{i=4}^{k-1} 2^{-\alpha_i} \right)},
\end{align*}
as required by~\eqref{eq:5}.

\medskip
\emph{Case 2: $n_1 \geq 17$.}
We distinguish two subcases, $n \leq 2^k +2$ and $n \geq 2^k +3$.

\smallskip
\emph{Case 2.a: $n \leq 2^k +2$.}
Since $n_1 \geq 17 \geq 3$ it follows that $n_2 \leq 2^k$ and thus the
inductive upper bound on $P(n_2)^{\frac{1}{n_2-1}}$ has a shorter
expansion (up to $k-2$):
\begin{align*}
P(n_2)^{\frac{1}{n_2-1}}  &\leq 677^{1/16} \exp{ \left( \sum_{i=4}^{k-2}
  2^{-\alpha_i} \right)},
\text{ or equivalently, } \\
P(n_2) &\leq 677^{\frac{n_2-1}{16}} \exp{ \left( (n_2-1) \sum_{i=4}^{k-2}
  2^{-\alpha_i} \right)}.
\end{align*}

Since $n_1 \leq n_2$, the same holds for $P(n_1)^{\frac{1}{n_1-1}}$:
\begin{align*}
P(n_1)^{\frac{1}{n_1-1}} &\leq 677^{1/16} \exp{ \left( \sum_{i=4}^{k-2}
  2^{-\alpha_i} \right)},
\text{ or equivalently, } \\
P(n_1) &\leq 677^{\frac{n_1-1}{16}} \exp{ \left( (n_1-1) \sum_{i=4}^{k-2}
  2^{-\alpha_i} \right)}.
\end{align*}

Since $n_1+n_2=n+1$, putting these two inequalities together yields:
\begin{align*}
P(n_1)  P(n_2) + 1
&\leq 677^{\frac{n-1}{16}} \exp{ \left( (n-1) \sum_{i=4}^{k-2} 2^{-\alpha_i} \right)} +1 \\
&\leq 677^{\frac{n-1}{16}} \exp{ \left( (n-1) \sum_{i=4}^{k-2} 2^{-\alpha_i} \right)}
\left( 1 + 677^{-\frac{n-1}{16}} \right) \\
&\leq 677^{\frac{n-1}{16}} \exp{ \left( (n-1) \sum_{i=4}^{k-2} 2^{-\alpha_i} \right)}
\exp{ \left( 677^{-\frac{n-1}{16}} \right)}. \\
\end{align*}

To settle Case 2.a, it suffices to show the following.
\begin{equation} \label{eq:3}
677^{\frac{n-1}{16}} \exp{ \left( (n-1) \sum_{i=4}^{k-2} 2^{-\alpha_i} \right)}
\exp{ \left( 677^{-\frac{n-1}{16}} \right)}
\leq 677^{\frac{n-1}{16}} \exp{ \left( (n-1) \sum_{i=4}^{k-1}
  2^{-\alpha_i} \right)}.
\end{equation}
Recall that $k \geq 5$ and this inequality is needed here; for $k \leq 4$,
the second factors on the left and the right side of \eqref{eq:3} are both equal to $1$,
and so~\eqref{eq:3} would not hold. Note that \eqref{eq:3} is equivalent to
\begin{align*}
\exp{ \left( (n-1) \sum_{i=4}^{k-2} 2^{-\alpha_i} \right)}
\exp{ \left( 677^{-\frac{n-1}{16}} \right)}
&\leq \exp{ \left( (n-1) \sum_{i=4}^{k-1} 2^{-\alpha_i} \right)},\\
\exp{ \left( 677^{-\frac{n-1}{16}} \right)}
&\leq \exp{ \left( (n-1) 2^{-\alpha_{k-1}} \right)},\\
677^{-\frac{n-1}{16}} &\leq (n-1) 2^{-\alpha_{k-1}}.
\end{align*}

Recall that $\alpha_{k-1}=2^{k-1}+k$; we also have $n-1 \geq 2^k$,
hence $ \frac{n-1}{2} \geq 2^{k-1}$. These relations yield
\begin{equation} \label{eq:14}
(n-1) 2^{-\alpha_{k-1}} = \frac{n-1}{2^{\alpha_{k-1}}} \geq \frac{2^k}{2^{\alpha_{k-1}}}=
    \frac{1}{2^{2^{k-1}}} \geq \frac{1}{2^{\frac{n-1}{2}}} \geq
    \frac{1}{677^{\frac{n-1}{16}}} ,
\end{equation}
as required.

\smallskip
\emph{Case 2.b:  $n \geq 2^k +3$.}
Assume that $2^{k_1}+1 \leq n_1 \leq 2^{k_1+1}$ for a suitable
$4 \leq k_1 \leq k$; indeed, $n_1 \geq 17$ implies $k_1 \geq 4$.
If we would have $k_1=k$ then $n_2 \geq n_1 \geq 2^k +1$ hence
$n_1 + n_2 \geq 2^{k+1} +2$, or $n \geq 2^{k+1} +1$, in contradiction
to the original assumption on $n$ in the theorem. It follows that $k_1 \leq k-1$,
and further that $n_1 \leq 2^{k_1+1} \leq 2^k$ and $n \geq 2^{k_1+1}+3$.
The inductive upper bound on $P(n_1)^{\frac{1}{n_1-1}}$ has the expansion:
\begin{align*}
P(n_1)^{\frac{1}{n_1-1}}   &\leq 677^{1/16} \exp{ \left( \sum_{i=4}^{k_1-1}
  2^{-\alpha_i} \right)},
\text{ or equivalently, } \\
P(n_1) &\leq 677^{\frac{n_1-1}{16}} \exp{ \left( (n_1-1) \sum_{i=4}^{k_1-1}
  2^{-\alpha_i} \right)}.
\end{align*}

By the inductive assumption we also have
\begin{align*}
P(n_2)^{\frac{1}{n_2-1}}  &\leq 677^{1/16} \exp{ \left( \sum_{i=4}^{k-1}
  2^{-\alpha_i} \right)},
\text{ or equivalently, } \\
P(n_2) &\leq 677^{\frac{n_2-1}{16}} \exp{ \left( (n_2-1) \sum_{i=4}^{k-1}
  2^{-\alpha_i} \right)}.
\end{align*}
Recall that $k_1 \geq 4$. Since $n_1+n_2=n+1$, putting these two
inequalities together yields:
\begin{align*}
P(n_1)  P(n_2) + 1
&\leq 677^{\frac{n-1}{16}} \exp{ \left( (n-1) \sum_{i=4}^{k_1-1} 2^{-\alpha_i}
+ (n_2-1) \sum_{i=k_1}^{k-1} 2^{-\alpha_i} \right)} +1 \\
&\leq  677^{\frac{n-1}{16}} \exp{ \left( (n-1) \sum_{i=4}^{k_1-1} 2^{-\alpha_i}
+ (n_2-1) \sum_{i=k_1}^{k-1} 2^{-\alpha_i} \right)} \left( 1 + 677^{-\frac{n-1}{16}} \right) \\
&\leq  677^{\frac{n-1}{16}} \exp{ \left( (n-1) \sum_{i=4}^{k_1-1} 2^{-\alpha_i}
+ (n_2-1) \sum_{i=k_1}^{k-1} 2^{-\alpha_i} \right)} \exp{ \left( 677^{-\frac{n-1}{16}} \right)}.
\end{align*}

To settle Case 2.b, it suffices to show that
\begin{align*}
& 677^{\frac{n-1}{16}} \exp{ \left( (n-1) \sum_{i=4}^{k_1-1} 2^{-\alpha_i}
+ (n_2-1) \sum_{i=k_1}^{k-1} 2^{-\alpha_i} \right)} \exp{ \left( 677^{-\frac{n-1}{16}} \right)}
\leq \\
& 677^{\frac{n-1}{16}} \exp{ \left( (n-1) \sum_{i=4}^{k-1} 2^{-\alpha_i} \right)},
\end{align*}
or equivalently,
\begin{align*}
\exp{ \left( (n_2-1) \sum_{i=k_1}^{k-1} 2^{-\alpha_i} \right)}
\exp{ \left( 677^{-\frac{n-1}{16}} \right)}
&\leq \exp{ \left( (n-1) \sum_{i=k_1}^{k-1} 2^{-\alpha_i} \right)},\\
\exp{ \left( 677^{-\frac{n-1}{16}} \right)}
&\leq \exp{ \left( (n-n_2) \sum_{i=k_1}^{k-1} 2^{-\alpha_i} \right)},\\
677^{-\frac{n-1}{16}} &\leq (n_1-1) \sum_{i=k_1}^{k-1} 2^{-\alpha_i}.
\end{align*}
Recall that $\alpha_{k_1}=2^{k_1}+k_1+1$ and that $n_1 -1 \geq 2^{k_1}$
by the assumption of Case 2.b; we also have (by the same reasons):
\begin{align*}
  n \geq 2^{k_1+1}+3 ~~~&\Rightarrow~~~ \frac{n-1}{2} \geq 2^{k_1}+1,\\
  k_1 \leq k-1 ~~~&\Rightarrow~~~ 2^{-\alpha_{k_1}} \leq \sum_{i=k_1}^{k-1} 2^{-\alpha_i}.
\end{align*}

From these relations we deduce that
\begin{equation} \label{eq:15}
\frac{1}{677^{\frac{n-1}{16}}} \leq
\frac{1}{2^{\frac{n-1}{2}}} \leq \frac{1}{2^{2^{k_1}+1}} =
\frac{2^{k_1}}{2^{\alpha_{k_1}}} \leq \frac{n_1-1}{2^{\alpha_{k_1}}}
\leq (n_1-1) \sum_{i=k_1}^{k-1} 2^{-\alpha_i},
\end{equation}
as required.
\end{proof}

\begin{corollary} \label{cor:1}
$P(n)=O(1.50284^n)$.
\end{corollary}
\begin{proof}
By Theorem~\ref{thm:recurrence} and Fact~\ref{fact:2} in Appendix~\ref{sec:numeric})
we obtain
\begin{equation*}
P(n)^{\frac{1}{n}}  \leq P(n)^{\frac{1}{n-1}}
\leq 677^{1/16} \exp{ \left( \sum_{i=4}^{k-1} 2^{-\alpha_i} \right)}
\leq  677^{1/16} \exp{ \left( \sum_{i=4}^{\infty} 2^{-\alpha_i} \right)}
\leq 1.50284.
\qedhere
\end{equation*}
\end{proof}

\paragraph{Proof of Theorem~\ref{thm:main}.}
 By Corollary~\ref{cor:1} we have $C_\x(n) = O(n P(n))$
 and by Theorem~\ref{thm:cx} we have $C(n) \leq (2n-5) \, C_\x(n)$.
  It follows that
$$ C(n) \leq (2n-5) \, C_\x(n) \leq 2n^2 \, P(n) \leq 2n^2 \cdot 1.50284^n = O(1.50285^n), $$
  as required.
  \qed

\section{Counting algorithm} \label{sec:algo}

The number of crossing-free structures (matchings, spanning trees,
spanning cycles, triangulations) on a set of $n$ points in the plane is known to be
exponential in $n$~\cite{DSST13,GNT00,RSW08,SS13,SSW12,SW06a}.
It is a challenging problem to determine the number of configurations faster than
listing all such configurations (\ie, count faster than enumerate)~\cite{ABRS15}.
Exponential-time algorithms have been developed for triangulations~\cite{AS13},
planar graphs~\cite{RW11}, and matchings~\cite{Wet14} that count these structures
exponentially faster than the number of structures.
Recently, it has been shown that the number of triangulations on $n$ points
in the plane can be counted in subexponential time~\cite{MM16},
and this result extends to counting noncrossing prefect matchings, spanning trees,
spanning cycles, 3-regular graphs, and more.

Here we show that given a plane straight-line graph $G$ with $n$ vertices
(\eg, a triangulation), convex polygons in $G$ can be counted in polynomial time.

\paragraph{Proof of Theorem~\ref{thm:algo}.}
Let $G=(V,E)$ be a plane straight line graph. For counting and enumerating
convex cycles (\ie, polygons) in $G$, we adapt a dynamic programming approach by
Eppstein~\etal~\cite{EORW92}, originally developed for finding
subsets in convex position of an $n$-element point set in the plane
optimizing various parameters, \eg, the area or the perimeter of the convex hull.
The dynamic program relies on the following observations:

\begin{enumerate} \itemsep 1pt
\item Assume, by rotating $G$ if necessary, that no two vertices have the same
$x$- or $y$-coordinates. Denote the vertices of $G$ by
$V=\{v_1,\ldots , v_n\}$, ordered by their $x$-coordinates.
Every convex polygon $\xi$ has a leftmost vertex $v_i$ and a rightmost vertex $v_k$.
The points $v_i$ and $v_k$ decompose $\xi$ into two $x$-monotone chains in $G$:
a convex (lower) chain and a concave (upper) chain connecting $v_i$ and $v_k$;
refer to Fig.~\ref{fig:algorithm}.
\item Conversely, the union of any convex and concave chains between vertices $v_i$ and
$v_k$ form a convex polygon, unless the two chains are equal, which happens
when both chains are the one-edge chain $v_iv_k$.
For $1\leq i<k\leq n$, denote by $\tcup(i,k)$ and $\tcap(i,k)$,
respectively, the number of $x$-monotone convex and concave chains between $v_i$ and $v_k$.
\item The number of convex polygons in $G$ whose leftmost and rightmost vertex, respectively,
are $v_i$ and $v_k$ is $\tcup(i,k) \cdot \tcap(i,k)$ if $v_iv_k\notin E$,
and $\tcup(i,k) \cdot \tcap(i,k)-1$ otherwise. Consequently, it is enough
to compute $\tcup(i,k)$ and $\tcap(i,k)$ for all $1\leq i<k\leq n$ in $O(n^2)$ time.
We consider $\tcup(i,k)$ only, the case of $\tcap(i,k)$ is analogous.
\end{enumerate}

\begin{figure}[hbtp]
\centering
\includegraphics[width=.39\textwidth]{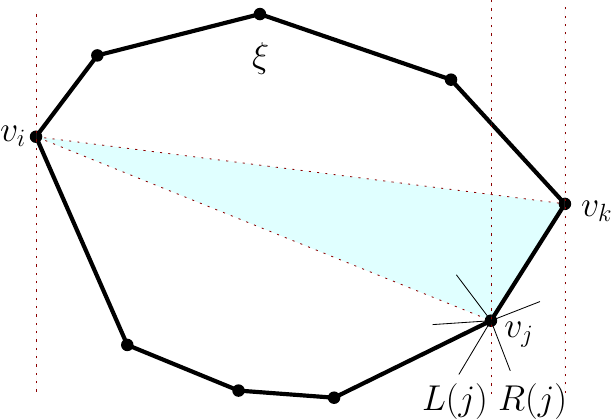}
\caption{A convex polygon where $v_i$ is the leftmost vertex, $v_k$ is the rightmost vertex
and $v_j$ is the penultimate vertex on the lower (convex) chain between $v_i$ and $v_k$.
The edges in $L(j)$ and $R(j)$, whose right and left endpoint, respectively, is $v_j$
are shown partially in the neighborhood of $v_j$.}
\label{fig:algorithm}
\end{figure}

We introduce a third parameter. For $1\leq i\leq j<k\leq n$ where $v_jv_k\in E$,
let $\tcup(i,j,k)$ be the number of $x$-monotone convex chain between $v_i$ and $v_k$
whose rightmost edge is $v_jv_k$. Since $G$ is planar, the number of triples
$1\leq i\leq j<k\leq n$, where $v_jv_k\in E$, is bounded from above by $|V|\cdot |E| = O(n^2)$.
Further, for each vertex $v_j\in V$, we partition the incident edges into two subsets:
let $L(j)$ be the set of indices $\ell$ such that $v_\ell v_j \in E$ and $v_\ell$
is the \emph{left} endpoint of $v_\ell v_j$;
and let $R(j)$ be the set of indices $k$ such that
$v_jv_k\in E$ and $v_k$ is the \emph{right} endpoint of $v_jv_k$.
Note that for all $1\leq i<k\leq n$ we have
\begin{equation}\label{eq:ijk}
\tcup(i,k)=\sum_{j\in L(k):i\leq j} \tcup(i,j,k).
\end{equation}

We compute the values $\tcup(i,j,k)$ in $O(n^2)$ time by dynamic programming;
the entries are computed in increasing order of $j-i$.
In a preprocessing step, we sort the indices $\ell\in L(j)$ by $\slope(v_\ell v_j)$,
and analogously the indices $k\in R(j)$ by $\slope(v_jv_k)$, for all $j=1,\ldots , n$.
This takes
$O(\sum_{j=1}^n \deg(v_j) \log \deg(v_j)) =O(\sum_{j=1}^n \deg(v_j)\log n) =O(n\log n)$ time.

Fix an index $i$, $1\leq i\leq n-1$.
For every $j\in \{i,\ldots, n-1\}$ and $k\in R(j)$, we compute
$\tcup(i,j,k)$ in $O(n)$ time in two nested loops.
For $j=i$ and all $k\in R(j)$, put $\tcup(i,j,k)=1$.
Consider next the values $j=i+1, \ldots, n$: if an $x$-monotone convex chain between
$v_i$ and $v_k$ contains the edge $v_jv_k$, and its penultimate edge is $v_\ell v_j$,
then $\slope(v_\ell v_j)<\slope(v_jv_k)$.
Consequently, for every $k\in R(j)$, we have
\begin{equation}\label{eq:slopes}
\tcup(i,j,k) = \sum_{\small \begin{array}{c}\ell\in L(j):
    i\leq \ell \mbox{ \rm and }\\ \slope(v_\ell v_j)<\slope(v_jv_k)\end{array}} \tcup(i,\ell,j).
\end{equation}
We can compute $\tcup(i,j,k)$ for all $k\in R(j)$ using \eqref{eq:slopes},
since $\tcup(i,\ell,j)$ for all $\ell=i,\ldots , j-1$ has already been computed.
If $k,k'\in R(j)$ and $\slope(v_jv_k)< \slope(v_jv_{k'})$, then
\begin{equation}\label{eq:slopes2}
\tcup(i,j,k') = \tcup(i,j,k)+ \sum_{\small \begin{array}{c}\ell\in L(j):
    i\leq \ell \mbox{ \rm and }\\
    \slope(v_jv_k< \slope(v_\ell v_j)<\slope(v_jv_{k'}) \end{array}} \tcup(i,\ell,j).
\end{equation}
Since $L(j)$ and $R(j)$ are ordered by $\slope(v_\ell v_j)$ and $\slope(v_jv_k)$, respectively,
we can use \eqref{eq:slopes2} to compute $\tcup(i,j,k)$ for all $k\in R(j)$ in this order
in $O(\deg(v_j))$ time. The running time for a fixed index $i$, $1\leq i\leq n-1$, is
$O(\sum_{j=i}^{n-1} \deg(v_j)) =O(\sum_{j=1}^n \deg(v_j)) = O(2|E|)=O(n)$.
The overall running time, over all $i=1,\ldots , n-1$, is $O(n^2)$, as claimed.

To enumerate all convex polygons in $G$, we compute the \emph{set} of convex and concave
chains corresponding to the values $\tcup(i,k)$ and $\tcap(i,k)$ for all pairs $(i,k)$,
$1\leq i<k\leq n$, where $\tcup(i,k)\cdot \tcap(i,k)\neq 0$ (the remaining pairs
do not contribute any convex polygons). For any such pair $(i,k)$, we enumerate all convex
polygons with leftmost vertex $v_i$ and rightmost vertex $v_k$ by reporting all combinations
of $x$-monotone convex chains and $x$-monotone concave chains between $v_i$ and $v_k$, except
for the possible combination of two single-edge chains when $v_iv_k\in E$.

To this end, we first prune the recursion tree on $\tcup(i,j,k)$ (resp., $\tcap(i,j,k)$)
induced by the recursion formula~\eqref{eq:slopes2}. Construct both recursion trees
(for $\tcup(i,j,k)$ and $\tcap(i,j,k)$).
In a top-down traversal, mark all nodes that contribute a nonzero value in the recursive
computation of $\tcup(i,k)$ and $\tcap(i,k)$, where $\tcup(i,k) \cdot \tcap(i,k)\neq 0$.
For these pairs $(i,k)$, $1\leq i<k\leq n$, we then compute the corresponding
\emph{set} of convex (resp., concave) chains by tracing back the recursion tree
and concatenating edges one-by-one in $O(1)$ time per edge (cf.~\cite[p.~387]{CLRS09}).
\qed

\paragraph{Acknowledgment.} The authors are grateful to an anonymous reviewer
for a very careful reading of the manuscript and for pertinent remarks.
In particular, the improvements suggested by the reviewer concerning the
counting algorithm in Section~\ref{sec:algo} were received with great appreciation.

\appendix

\section{Numeric calculations} \label{sec:numeric}

We need the following numerical estimates.

\begin{fact} \label{fact:1}
The following inequality holds:
$$ 2^{1/2} \exp{ \left( \sum_{i=1}^\infty 2^{-\alpha_i} \right)} \leq 1.5180. $$
\end{fact}
\begin{proof}
An easy calculation yields an upper bound
on the sum $\sum_{i=1}^\infty 2^{-\alpha_i}$:
\begin{align*}
\sum_{i=1}^{\infty} 2^{-\alpha_i}
&= 2^{-4} + 2^{-7} + 2^{-12} + 2^{-21} + \ldots \\
&\leq 2^{-4} + 2^{-7} + \sum_{i=1}^{\infty} 2^{-11-i} =
2^{-4} + 2^{-7} + 2^{-11}.
\end{align*}
It follows that
$$ 2^{1/2} \exp{ \left( \sum_{i=1}^\infty 2^{-\alpha_i} \right)}
\leq 2^{1/2} \exp{ \left( 2^{-4} + 2^{-7} + 2^{-11} \right) } \leq 1.5180, $$
as required.
\end{proof}

\begin{fact} \label{fact:2}
The following inequality holds:
$$ 677^{1/16} \exp{ \left( \sum_{i=4}^\infty 2^{-\alpha_i} \right)} \leq 1.50284. $$
\end{fact}
\begin{proof}
Similarly to the proof of Fact~\ref{fact:1},
an easy calculation yields an upper bound
on the sum $\sum_{i=4}^\infty 2^{-\alpha_i}$:
$$
\sum_{i=4}^{\infty} 2^{-\alpha_i} = 2^{-21} + 2^{-38} + \ldots
\leq \sum_{i=1}^{\infty} 2^{-20-i} = 2^{-20}.
$$
It follows that
$$ 677^{1/16} \exp{ \left( \sum_{i=4}^\infty 2^{-\alpha_i} \right)}
\leq 677^{1/16} \exp{ \left( 2^{-20} \right) } \leq 1.50284, $$
as required.
\end{proof}

\begin{fact} \label{fact:basis1}
The following holds:
$$ \max_{2\leq n\leq 16} P(n)^{\frac{1}{n-1}} = P(9)^{1/8} =26^{1/8} =1.50269\ldots $$
\end{fact}
\begin{proof}
Using the values of $P(n)$ from recurrence~\eqref{eq:1}, we
verify the following inequalities:
\begin{align*}
P(2) &=1 \text{ and } P(2)^{1/1} =1 \leq 26^{1/8} =1.50269\ldots\\
P(3) &=2 \text{ and } P(3)^{1/2} = 2^{1/2} = 1.4142\ldots \leq 26^{1/8} =1.50269\ldots\\
P(4) &=3 \text{ and } P(4)^{1/3} =3^{1/3} = 1.4422\ldots \leq 26^{1/8} =1.50269\ldots\\
P(5) &=5 \text{ and } P(5)^{1/4} =5^{1/4} = 1.4953\ldots \leq 26^{1/8} =1.50269\ldots\\
P(6) &=7 \text{ and } P(6)^{1/5} =7^{1/5} = 1.4757\ldots \leq 26^{1/8} =1.50269\ldots\\
P(7) &=11 \text{ and } P(7)^{1/6} =11^{1/6} = 1.4913\ldots \leq 26^{1/8} =1.50269\ldots\\
P(8) &=16 \text{ and } P(8)^{1/7} =16^{1/7} = 1.4859\ldots \leq 26^{1/8} =1.50269\ldots\\
P(9) &=26 \text{ and } P(9)^{1/8} =26^{1/8} =1.50269\ldots\\
P(10) &=36 \text{ and } P(10)^{1/9} =36^{1/9} =1.4890\ldots \leq 26^{1/8} =1.50269\ldots\\
P(11) &=56 \text{ and } P(11)^{1/10} =56^{1/10} =1.4956\ldots \leq 26^{1/8} =1.50269\ldots\\
P(12) &=81 \text{ and } P(12)^{1/11} =81^{1/11} =1.4910\ldots \leq 26^{1/8} =1.50269\ldots\\
P(13) &=131 \text{ and } P(13)^{1/12} =131^{1/12} =1.5012\ldots \leq 26^{1/8} =1.50269\ldots\\
P(14) &=183 \text{ and } P(14)^{1/13} =183^{1/13} =1.4929\ldots \leq 26^{1/8} =1.50269\ldots\\
P(15) &=287 \text{ and } P(15)^{1/14} =287^{1/14} =1.4981\ldots \leq 26^{1/8} =1.50269\ldots\\
P(16) &=417 \text{ and } P(16)^{1/15} =417^{1/15} =1.4951\ldots \leq 26^{1/8} =1.50269\ldots
\tag*{\qedhere}
\end{align*}
\end{proof}

\begin{fact} \label{fact:basis2}
The following holds:
$$ \max_{17\leq n\leq 32} P(n)^{\frac{1}{n-1}} = P(17)^{1/16} =677^{1/16} =1.50283\ldots $$
\end{fact}
\begin{proof}
Using the values of $P(n)$ from recurrence~\eqref{eq:1}, we
verify the following inequalities:
\begin{align*}
P(17) &=677 \text{ and } P(17)^{1/16} =677^{1/16} =1.50283\ldots\\
P(18) &=937 \text{ and } P(18)^{1/17} =937^{1/17} =1.4955\ldots \leq 677^{1/16} =1.50283\ldots\\
P(19) &=1457 \text{ and } P(19)^{1/18} =1457^{1/18} =1.4988\ldots \leq 677^{1/16} =1.50283\ldots\\
P(20) &=2107 \text{ and } P(20)^{1/19} =2107^{1/19} =1.4959\ldots \leq 677^{1/16} =1.50283\ldots\\
P(21) &=3407 \text{ and } P(21)^{1/20} =3407^{1/20} =1.5018\ldots \leq 677^{1/16} =1.50283\ldots\\
P(22) &=4759 \text{ and } P(22)^{1/21} =4759^{1/21} =1.4966\ldots \leq 677^{1/16} =1.50283\ldots\\
P(23) &=7463 \text{ and } P(23)^{1/22} =7463^{1/22} =1.4998\ldots \leq 677^{1/16} =1.50283\ldots\\
P(24) &=10843 \text{ and } P(24)^{1/23} =10843^{1/23} =1.4977\ldots \leq 677^{1/16} =1.50283\ldots\\
P(25) &=17603 \text{ and } P(25)^{1/24} =17603^{1/24} =1.5027\ldots \leq 677^{1/16} =1.50283\ldots\\
P(26) &=24373 \text{ and } P(26)^{1/25} =24373^{1/25} =1.4978\ldots \leq 677^{1/16} =1.50283\ldots\\
P(27) &=37913 \text{ and } P(27)^{1/26} =37913^{1/26} =1.5000\ldots \leq 677^{1/16} =1.50283\ldots\\
P(28) &=54838 \text{ and } P(28)^{1/27} =54838^{1/27} =1.4980\ldots \leq 677^{1/16} =1.50283\ldots\\
P(29) &=88688 \text{ and } P(29)^{1/28} =88688^{1/28} =1.5021\ldots \leq 677^{1/16} =1.50283\ldots\\
P(30) &=123892 \text{ and } P(30)^{1/29} =123892^{1/29} =1.4983\ldots \leq 677^{1/16} =1.50283\ldots\\
P(31) &=194300 \text{ and } P(31)^{1/30} =194300^{1/30} =1.5006\ldots \leq 677^{1/16} =1.50283\ldots\\
P(32) &=282310 \text{ and } P(32)^{1/31} =282310^{1/31} =1.4990\ldots \leq 677^{1/16} =1.50283\ldots
\tag*{\qedhere}
\end{align*}
\end{proof}

\begin{fact}\label{fact:Case1}
For $2 \leq n \leq 16$, we have
\begin{equation} \label{eq:12+}
P(n) \, \exp{ \left( (P(n))^{-1} \, 677^{-\frac{33-n}{16}} \right)}
\leq 677^{\frac{n-1}{16}}.
\end{equation}
\end{fact}
\begin{proof}
Let
\begin{equation} \label{eq:16}
  x_n =(P(n))^{-1} \, 677^{\frac{n-1}{16}}, \text{ for } 2 \leq n \leq 16.
\end{equation}
Then~\eqref{eq:12+} is equivalent to
\begin{equation} \label{eq:17}
\exp{ \left( \frac{x_n}{677^2} \right)} \leq x_n,  \text{ for } 2 \leq n \leq 16.
\end{equation}

By Fact~\ref{fact:basis1}, we have
$$ P(n)^{\frac{1}{n-1}} \leq P(9)^{1/8} =26^{1/8},  \text{ for } 2 \leq n \leq 16.
$$
and this implies
\begin{equation*}
x_n =(P(n))^{-1} \, 677^{\frac{n-1}{16}} \geq
\frac{677^{\frac{n-1}{16}}}{26^{\frac{n-1}{8}}} =
\left( \frac{677}{676} \right) ^{\frac{n-1}{16}}
\geq \left( \frac{677}{676} \right) ^{\frac{1}{16}} = 1.00009\ldots,
\text{ for } 2 \leq n \leq 16.
\end{equation*}
Obviously, we also have $x_n \leq 677$, for $n=2,\ldots,16$,
thus $x_n$ is bounded as follows:
$$ \left( \frac{677}{676} \right) ^{\frac{1}{16}} \leq x_n
\leq 677, \text{ for } 2 \leq n \leq 16.
$$

To verify~\eqref{eq:17}, we distinguish two cases:

\medskip
\emph{Case 1: $x_n \in \left[ \left( \frac{677}{676} \right) ^{\frac{1}{16}}, 2 \right].$}
Then
$$ \exp{ \left( \frac{x_n}{677^2} \right)} \leq
\exp{ \left( \frac{2}{677^2} \right)} =1.0000043\ldots \leq
\left( \frac{677}{676} \right) ^{\frac{1}{16}} = 1.00009\ldots
\leq x_n, $$
as required by~\eqref{eq:17}.

\medskip
\emph{Case 2: $x_n \in [2,677]$.}
Then
$$ \exp{ \left( \frac{x_n}{677^2} \right)} \leq
\exp{ \left( \frac{677}{677^2} \right)} =
\exp{ \left( \frac{1}{677} \right)} =1.0014\ldots \leq 2 \leq x_n, $$
as required by~\eqref{eq:17}.
\end{proof}


\begin{thebibliography}{99}

\bibitem{AHV+06} O. Aichholzer, T. Hackl, B. Vogtenhuber, C. Huemer,
F. Hurtado, and H. Krasser,
On the number of plane geometric graphs,
\emph{Graphs and Combinatorics} \textbf{23(1)} (2007), 67--84.

\bibitem{ACNS82} M. Ajtai, V. Chv\'atal, M. Newborn, and E. Szemer\'edi,
Crossing-free subgraphs,
\emph{Annals of Discrete Mathematics} \textbf{12} (1982), 9--12.

\bibitem{ABRS15}
V.~Alvarez, K.~Bringmann, S.~Ray, and R.~Seidel,
Counting triangulations and other crossing-free structures approximately,
\emph{Computational Geometry: Theory \& Applications} {\bf 48(5)} (2015), 386--397.
%\emph{Comput. Geom. Theory Appl.}  {\bf 48(5)} (2015), 386--397.

\bibitem{AS13}
V.~Alvarez and R.~Seidel,
A simple aggregative algorithm for counting triangulations of planar
point sets and related problems,
in \emph{Proc. 29th Sympos.on Comput. Geom.} (SOCG), ACM Press, 2013, pp.~1--8.

\bibitem{AKLS11}
B. Aronov, M. van Kreveld, M. L\"offler, and R. I. Silveira,
Peeling meshed potatoes, \emph{Algorithmica} \textbf{60(2)} (2011), 349--367.

\bibitem{BKK+07}
K. Buchin, C. Knauer, K. Kriegel, A. Schulz, and R. Seidel.
On the number of cycles in planar graphs,
in \emph{Proc. 13th Annual International Conference on Computing
and Combinatorics} (COCOON), LNCS~4598, Springer, 2007, pp.~97--107.

\bibitem{CY86}
J. S. Chang and C. K. Yap, A polynomial solution for the potato-peeling problem,
\emph{Discrete \& Computational Geometry} \textbf{1} (1986), 155--182.

\bibitem{CLRS09} T.~H.~Cormen, C.~E.~Leiserson, R.~L.~Rivest, and C.~Stein,
\emph{Introduction to Algorithms},
3rd~edition, MIT Press, Cambridge, 2009.

\bibitem{DLST16}
A. Dumitrescu, M. L\"offler, A. Schulz, and Cs. D. T\'oth,
Counting carambolas,
\emph{Graphs and Combinatorics} \textbf{32(3)} (2016), 923--942.

\bibitem{DMT16}
A. Dumitrescu, R. Mandal, and Cs. D. T\'oth,
Monotone paths in geometric triangulations,
Preprint, \href{https://arxiv.org/abs/1608.04812}{arXiv:1608.04812}, 2016.
An extended abstract of an earlier version in \emph{Proc. 27th International Workshop
  on Combinatorial Algorithms (IWOCA 2016)}, LNCS~9843, pp.~411--422, Springer, 2016.

\bibitem{DSST13}
A. Dumitrescu, A. Schulz, A. Sheffer, and Cs. D. T\'oth,
Bounds on the maximum multiplicity of some common geometric graphs,
\emph{SIAM Journal on Discrete Mathematics} \textbf{27(2)} (2013), 802--826.

\bibitem{DT12}
A. Dumitrescu and Cs. D. T\'oth,
Computational Geometry Column 54,
\emph{SIGACT News Bulletin} \textbf{43(4)} (2012), 90--97.

\bibitem{EORW92}
D.~Eppstein, M.~Overmars, G.~Rote, and G.~Woeginger,
Finding minimum area $k$-gons,
\emph{Discrete \& Computational Geometry} \textbf{7(1)} (1992), 45--58.

\bibitem{Erd78}
P. Erd\H{o}s, Some more problems on elementary geometry,
%\emph{Austral. Math. Soc. Gaz.} \textbf{5} (1978), 52--54.
\emph{Gazette of the Australian Mathematical Society} \textbf{5} (1978), 52--54.

\bibitem{ESz35}
P.~Erd\H{o}s and G.~Szekeres, A combinatorial problem in geometry,
\emph{Compositio Mathematica} \textbf{2} (1935), 463--470.

\bibitem{GNT00} A. Garc\'{\i}a, M. Noy and A. Tejel,
Lower bounds on the number of crossing-free subgraphs of $K_N$,
\emph{Computational Geometry: Theory \& Applications} \textbf{16(4)} (2000), 211--221.

\bibitem{Goo81}
J. E. Goodman,
On the largest convex polygon contained in a non-convex $n$-gon or how to peel a potato,
\emph{Geometria Dedicata} \textbf{11} (1981), 99--106.

\bibitem{KLP12}
M. J. van Kreveld, M. L{\"o}ffler, and J. Pach,
How many potatoes are in a mesh?,
in \emph{Proc. 23rd Internat. Sympos. on Alg. Comput.}
(ISAAC), LNCS~7676, Springer, 2012, pp.~166--176.

\bibitem{MM16}
D. Marx and T. Miltzow,
Peeling and nibbling the cactus:
Subexponential-time algorithms for counting triangulations and related problems,
in \emph{Proc. 32nd International Symposium on Computational Geometry}
(SoCG), LIPIcs~51, Schloss Dagstuhl, 2016, article 52.

\bibitem{MS00}
W. Morris and V. Soltan,
The Erd\H{o}s-Szekeres problem on points in convex position---a survey,
\emph{Bulletin of AMS} \textbf{37} (2000), 437--458.

\bibitem{RSW08}
A.~Razen, J.~Snoeyink, and E.~Welzl,
Number of crossing-free geometric graphs vs. triangulations,
\emph{Electronic Notes in Discrete Mathematics} \textbf{31} (2008), 195--200.

\bibitem{RW11}
A.~Razen and E.~Welzl,
Counting plane graphs with exponential speed-up,
in \emph{Rainbow of Computer Science}, Springer, 2011, pp.~36--46.

\bibitem{SS11}
M. Sharir and A. Sheffer,
Counting triangulations of planar point sets,
{\em The Electronic Journal of Combinatorics} {\bf 18} (2011), P70.

\bibitem{SS13}
M. Sharir and A. Sheffer,
Counting plane graphs: cross-graph charging schemes,
\emph{Combinatorics, Probability {\&} Computing} \textbf{22(6)} (2013), 935--954.

\bibitem{SSW12}
M. Sharir, A. Sheffer, and E. Welzl,
Counting plane graphs: perfect matchings, spanning cycles, and
Kasteleyn's technique,
\emph{Journal of Combinatorial Theory, Ser.~A} \textbf{120(4)} (2013), 777--794.

\bibitem{SW06a} M. Sharir and E. Welzl,
On the number of crossing-free matchings, cycles, and partitions,
\emph{SIAM Journal on Computing} {\bf 36(3)} (2006), 695--720.

\bibitem{She14} A. Sheffer,
Numbers of plane graphs,
\href{https://adamsheffer.wordpress.com/numbers-of-plane-graphs/}
     {https://adamsheffer.wordpress.com/numbers-of-plane-graphs/}
(version of May, 2015).

\bibitem{Wet14} M. Wettstein,
Counting and enumerating crossing-free geometric graphs,
in \emph{Proc. 30th Sympos. on Comput. Geom.} (SOCG), ACM Press, 2014, pp.~1--10.
Full paper available at \href{https://arxiv.org/abs/1604.05350}{arXiv:1604.05350}, 2016.

\end{thebibliography}
\end{document}